\documentclass{compositio}

\usepackage[centertags]{amsmath}
\usepackage[all]{xy}
\usepackage{amsfonts}
\usepackage{amssymb}
\usepackage{amsthm}
\usepackage{amscd}
\usepackage{hyperref}

\renewcommand{\contentsname}{Contents}

\newcommand{\cO}{\mathcal{O}}

\newcommand{\h}{\mathfrak{h}}

\newcommand{\p}{\mathfrak{p}}

\newcommand{\cL}{\mathcal{L}}
\newcommand{\cV}{\mathcal{V}}

\newcommand{\cX}{\mathcal{X}}
\newcommand{\ds}{\displaystyle}

\newcommand{\cE}{\mathcal{E}}

\newcommand{\ra}{\rightarrow}
\newcommand{\xra}[1]{\xrightarrow{#1}}
\newcommand{\hra}{\hookrightarrow}

\newcommand{\mtwo}[4]{\left(
        \begin{matrix}#1&#2\\#3&#4
        \end{matrix}\right)}

\newcommand{\vphi}{\varphi}

\newcommand{\Zhat}{\widehat{\Z}}

\newcommand{\comment}[1]{}
\newcommand{\Q}{\mathbb{Q}}

\newcommand{\cC}{\mathcal{C}}

\newcommand{\C}{\mathbb{C}}

\newcommand{\Qbar}{\overline{\Q}}

\newcommand{\Z}{\mathbb{Z}}
\newcommand{\F}{\mathbb{F}}

\newcommand{\Fbar}{\overline{\F}}

\newcommand{\NN}{\mathbb{N}}

\newcommand{\n}{\mathfrak n}
\newcommand{\isom}{\cong}
\newcommand{\Af}{\mathbb{A}_f}
\newcommand{\bc}{\mathfrak{c}}
\newcommand{\f}{\textrm{f}}
\newcommand{\g}{\textrm{g}}
\newcommand{\lambar}{\overline{\lambda}}

\DeclareFontEncoding{OT2}{}{} 

\renewcommand{\H}{\HH}

\newcommand{\cI}{\mathcal I}

\DeclareMathOperator{\chr}{char}
\DeclareMathOperator{\alg}{alg}

\DeclareMathOperator{\lcm}{lcm}

\DeclareMathOperator{\tors}{tors}
\DeclareMathOperator{\Pic}{Pic}

\DeclareMathOperator{\ab}{ab}
\DeclareMathOperator{\Aut}{Aut}

\DeclareMathOperator{\ord}{ord}
\DeclareMathOperator{\GL}{GL}

\DeclareMathOperator{\PGL}{PGL}
\DeclareMathOperator{\Gal}{Gal}
\DeclareMathOperator{\SL}{SL}

\DeclareMathOperator{\HH}{H}
\DeclareMathOperator{\Ram}{Ram}
\DeclareMathOperator{\Res}{Res}

\DeclareMathOperator{\red}{red}
\DeclareMathOperator{\nr}{nr}
\DeclareMathOperator{\Stab}{Stab}

\DeclareMathOperator{\End}{End}
\DeclareMathOperator{\Hom}{Hom}
\DeclareMathOperator{\rec}{rec}

\DeclareMathOperator{\dist}{dist}

\DeclareMathOperator{\Lie}{Lie}
\DeclareMathOperator{\Cl}{Cl}
\DeclareMathOperator{\RED}{RED}

\DeclareMathOperator{\Spec}{Spec}

\theoremstyle{plain} 
\newtheorem{thm}{Theorem}[section] 

\newtheorem{cor}[thm]{Corollary}
\newtheorem{lem}[thm]{Lemma}

\theoremstyle{definition} 
\theoremstyle{remark} 
\newtheorem{rem}{Remark}

\newcommand{\authnote}[2][]{\noindent {\if!#1!  {\bf TODO} \else {\small \bf #1} \fi: #2} \vspace{0.1in}}

\newcounter{tasknumber}
\newcommand{\task}[2][]{%
  \addtocounter{tasknumber}{1}%
  \begin{center}%
  \framebox[1.1\width]{\begin{minipage}{0.9\textwidth}%
  \textbf{Task \arabic{tasknumber}} \textit{\if!#1!(unassigned)\else (#1)\fi}: {#2}%
  \end{minipage}}%
  \end{center}%
}

\begin{abstract}
We construct liftings of reduction maps from CM points to supersingular points for general quaternion algebras and use these liftings to establish a precise correspondence between CM points on indefinite quaternion algebras with a given conductor and CM points on certain corresponding totally definite quaternion algebras. 
\end{abstract}

\begin{document}


\title[Reduction maps]
{Liftings of Reduction Maps for Quaternion Algebras}
\author{Christophe Cornut}
\email{cornut@math.jussieu.fr}
\address{Institut de Math\'ematiques de Jussieu, 175 rue du Chevaleret, 75013 Paris, France}

\author{Dimitar Jetchev}
\email{dimitar.jetchev@epfl.ch}
\address{Ecole Polytechnique F\'ed\'erale de Lausanne\\Lausanne\\Switzerland}
%
\thanks{}

\maketitle


%
%
\section{Introduction}

Let $F$ be a totally real number field and let $B$ be a quaternion algebra over $F$. Let $G = \Res_{F/\Q}B^\times$ be the algebraic group over $\Q$ associated to $B^\times$. Let $\mathbb{A}_f$ be the finite adeles of $\Q$. 

For an open subgroup $H \subset G(\Af)$ that is compact modulo the center let 
\[
\mathcal{X}(G,H) =G(\Q)\backslash G(\Af)/H.
\]
This is a finite set. In the case when $H=\widehat{R}^{\times}$ for
some $\mathcal{O}_{F}$-order $R$ in $B$, the map $\widehat{b}\mapsto(\widehat{b}\cdot\widehat{R})\cap B$
identifies $\mathcal{X}(G,H)$ with the set of $B^{\times}$-homothety
classes of locally principal fractional right $R$-ideals in $B$.
When $B$ is not totally definite, the reduced norm $\nr \colon B\rightarrow F$
induces a bijection 
\[
\mathcal{X}(G,H)\stackrel{\simeq}{\longrightarrow}\mathrm{nr}(B^{\times})\backslash\widehat{F}^{\times}/\mathrm{nr}(H)
\]
by the strong approximation theorem \cite[p.81]{vigneras:quaternion}. Moreover, the norm theorem \cite[Thm.III.4.1]{vigneras:quaternion} implies that $\mathrm{nr}(B^{\times})$
is precisely the subgroup of elements $\lambda\in F^{\times}$ such that $\lambda_{v}>0$
for all places $v\mid\infty$ of $F$ where $B$ is not split, i.e., 
where $B_{v}\not\simeq M_{2}(F_{v})$. On the other hand, if $B$ is
totally definite, then $\mathcal{X}(G,H)$ is a genuinely non-commutative
object, which shows up quite often in low dimensional arithmetic geometry (see, e.g., \cite{gross:heights} and \cite{ribet:modreps}). 

Next, let $K$ be a totally imaginary quadratic extension of $F$ that embeds into $B$. Fix an embedding $K\hookrightarrow B$. If $T = \Res_{K/\Q} K^\times$ is the associated rational torus, we get a corresponding  embedding $T \hra G$. Define  
\[
\mathcal{CM}(G,H)=T(\Q)\backslash G(\Af)/H.
\]
This is now an infinite set with an obvious projection map
\begin{equation}\label{eq:projred}
\pi \colon \mathcal{CM}(G,H)\rightarrow\mathcal{X}(G,H)
\end{equation}
In addition, it is equipped with a left action
of $T(\Af)$ with finite orbits. Using Artin's reciprocity
map $\rec_{K} \colon T(\Af) = \widehat{K}^{\times}\twoheadrightarrow\Gal_{K}^{ab}$ 
whose kernel equals the closure of $K^{\times}$ in $\widehat{K}^{\times}$, 
we can also view this action as a continous action of $\Gal_{K}^{ab}$. We shall thus refer to it as the 
\emph{Galois action} on $\mathcal{CM}(G, H)$. 

\vspace{0.1in}

\noindent As suggested by the notation, such sets are most frequently occurring
as sets of complex multiplication points, or special points, on
certain Shimura varieties. Assume for instance that $B$ is split
at a single place $v\mid\infty$ of $F$, say $v=v_{\iota}$ corresponding
to an embedding $\iota:F\rightarrow\mathbb{R}$. Let $X$ be the Shimura
curve of level $H$ attached to $B$. It is an algebraic curve over
the reflex field $\iota F$, and 
\[
\mathcal{CM}(G,H) \simeq G(\Q)\backslash\left(G(\Af)/H\times G(\Q)\cdot\tau\right) \subset X(\mathbb{C})=G(\Q)\backslash\left(G(\Af)/H\times(\mathbb{C}-\mathbb{R})\right)
\]
is the set of special points with complex multiplication by $K$ in
$X$. In this formula, we have fixed an isomorphism $B\otimes_{F,\iota}\mathbb{R}\simeq M_{2}(\mathbb{R})$
to define the action of $G(\Q)$ on $\mathbb{C}-\mathbb{R}$,
and $\tau$ is the unique point on the upper half-plane whose stabilizer is $K^\times = T(\Q)$ under this isomorphism. These special points are defined over the maximal abelian extension of $\iota K$ in $\mathbb{C}$, and the above Galois action is the natural one (this follows from the definition of Shimura varieties). In this setting, the projection map is almost
the projection to the set of connected components of $X$, namely
\[
\mathcal{CM}(G,H)\hookrightarrow X(\mathbb{C})\twoheadrightarrow\pi_{0}(X(\mathbb{C}))=G(\Q)\backslash\left(G(\Af)/H\times\{\pm1\}\right)=G(\Q)^+ \backslash G(\Af) / H.\]
Here $G(\Q)$ acts on $\{\pm1\}$ by the sign of $\iota\circ\det \colon G(\Q) \rightarrow\mathbb{R}^{\times}$,
and $G(\Q)^+$ is the kernel of this action.

\paragraph{Reduction maps.} 
There are similar but more interesting maps to consider. For instance,
let $v$ be a finite place of $F$ such that $B_{v}\simeq M_{2}(F_v)$.
Then our Shimura curve has a natural model over the corresponding
local ring, the special fiber of which contains a distinguished
finite set of supersingular points and this finite set of
points may be identified with a set $\mathcal{X}(G',H')$ as above,
where $G'$ is the algebraic group over $\mathbb{Q}$ associated to
the totally definite quaternion algebra $B'$ over $F$ which is obtained
from $B$ by changing the invariants at $v_{\iota}$ and $v$ (see, e.g., \cite{deligne-rapoport} and \cite{katzmazur} for 
$B=M_{2}(\mathbb{Q})$, and \cite{carayol} for
the remaining cases). If we choose an extension of $v$ to the maximal
abelian extension of $\iota K$ in $\mathbb{C}$ then each
CM point extends to the corresponding local ring. Moreover, if $v$
does not split in $K$, these extended points reduce to supersingular
points on the special fiber and one obtains a reduction map
\begin{equation}\label{eq:red}
\red \colon \mathcal{CM}(G,H)\rightarrow\mathcal{X}(G',H')
\end{equation}
that is described in \cite{cornut:inventiones} and Section~\ref{sec:apps} if $B=M_{2}(\mathbb{Q})$ and in \cite{cornut-vatsal} for the remaining cases. On the other hand, if $B$ is ramified at
$v$, there is a similar description for the reduction map with
values in the set of irreducible components of the special fiber of
a suitable model of $X$ via Ribet's theory of bimodules. We refer
to \cite{molina} for a survey of such maps, and to \cite{molina:hyperelliptic} for
applications to the determination of explicit equations for certain
hyperelliptic Shimura curves.

Given such a reduction map, it is expected that the Galois orbits
in $\mathcal{CM}(G,H)$ tend to be equidistributed among the finitely many fibers of the reduction map 
$\mathrm{red}$. Such equidistribution results are already known in various cases 
\cite{michel:subconvexity}, \cite{cornut-vatsal}, \cite{michel-harcos}, \cite{jetchev-kane}, \cite{molina} and 
were crucial in the proof of Mazur's non-vanishing conjecture by the first author and Vatsal 
\cite{cornut:inventiones}, \cite{vatsal:uniform}, \cite{cornut-vatsal:durham}. 

\paragraph{Our contributions.} 
We propose a simple strategy to reduce the study of the arithmetically
interesting reduction maps \eqref{eq:red} to that of the more straightforward
projections \eqref{eq:projred}. In all cases, there is indeed a natural 
$\mathrm{Gal}_{K}$-equivariant map 
\[
\theta \colon \mathcal{CM}(G,H)\rightarrow\mathcal{CM}(G',H')
\]
such that $\mathrm{red}=\pi\circ\theta$. Thus, for any $\mathrm{Gal}_{K}$-orbit
$\Gamma$ in $\mathcal{CM}(G,H)$ and any point $s\in\mathcal{X}(G',H')$,
we obtain a $\kappa$-to-$1$ surjective map
\[
\theta \colon \Gamma\cap\mathrm{red}^{-1}(s)\rightarrow\Gamma'\cap\pi^{-1}(s)
\]
where $\Gamma'=\theta(\Gamma)$ is a $\mathrm{Gal}_{K}$-orbit in
$\mathcal{CM}(G',H')$ and $\kappa=\kappa(\Gamma)=\left|\Gamma\right|/\left|\Gamma'\right|$. 

This paper essentially implements this strategy when $H=\widehat{R}^{\times}$
for some Eichler $\mathcal{O}_{F}$-order $R$ in $B$. The algebraic
description of $\theta$ (and also of $\mathrm{red}=\pi\circ\theta$)
is given in Section~\ref{subsec:adelic} in a more general setting following the
conventions of \cite{cornut-vatsal}. The size of the Galois orbits (and thus also
the constant $\kappa$ above) is controlled by a simple invariant that is defined in Section~\ref{subsec:cond} together with a refinement: these are the coarse and fine conductors $\bc_{g}$ and $\bc_{f}$, respectively. The number of Galois orbits with a prescribed
fine conductor is also given there using elementary local computations
that are carried on in Section~\ref{sec:ql}. Our main result, Theorem~\ref{thm:main}, 
then describes the restriction of $\theta$ to the fibers of the fine
conductor map. If we furthermore restrict $\theta$ to a given fiber
of $\mathrm{red}$ and $\pi$, we obtain an explicit correspondence
between certain sets of CM points on two distinct quaternion algebras,
a special case of which is used in \cite{jetchev-kane} and thoroughly explained
in the final section. Another case of our main theorem is used in \cite{molina:hyperelliptic} and \cite{molina} to compute explicit equations for Shimura curves with no cusps. 

\section{Review of quadratic orders}\label{sec:ql}
Only for this particular section, $\mathcal{O}_{F}$ will be a Dedekind domain with fraction field $F$. Let $K$ be a semi-simple commutative $F$-algebra of dimension $2$, i.e., $K\simeq F\times F$ or $K$ is a quadratic field extension of $F$. 
Let $\mathcal{O}_{K}$ be the integral closure of $\mathcal{O}_{F}$ in $K$. The map which sends $\mathfrak c$ to $\mathcal{O}_{\mathfrak c}=\mathcal{O}_{F}+\mathfrak c\mathcal{O}_{K}$ is a bijection from the set of all non-zero ideals $\mathfrak  c \subset\mathcal{O}_{F}$ onto the set of all $\mathcal{O}_{F}$-orders in $K$. It is well-known that all such orders are \emph{Gorenstein} rings. We refer to $\mathfrak c$ as the conductor of $\mathcal{O}=\mathcal{O}_{\mathfrak c}$. 

\subsection{Quadratic lattices}\label{subsec:quadlat}
Fix a free, rank one $K$-module $V$. Let $\mathcal{L}$ be the set of all full $\mathcal{O}_{F}$-lattices in $V$. The \emph{conductor} of a lattice $\Lambda \in\mathcal{L}$ is the conductor $c(\Lambda)$ of the $\mathcal{O}_{F}$-order
$\mathcal{O}(\Lambda)=\{\lambda\in K:\lambda \Lambda \subset \Lambda\}=\mathcal{O}_{c(\Lambda)}$. 
It follows from~\cite[Prop.7.2]{bass:gorenstein} that $\Lambda$ is a projective rank
one $\mathcal{O}_{c(\Lambda)}$-module. Let $[\Lambda]$ be its isomorphism class in the Picard group $\Pic(\cO_{c(\Lambda)})$. Since any two $\cO_F$-lattices in $V$ are $K^\times$-homothetic precisely when they have the same conductor and define the same class in the relevant Picard group, the map $\Lambda \mapsto [\Lambda]$ induces a bijection
\begin{equation}
K^{\times}\backslash\mathcal{L}\simeq \bigsqcup_{\mathfrak c}\mathrm{Pic}(\mathcal{O}_{\mathfrak c}).  
\end{equation} 
When $\mathcal{O}_{F}$ is a local ring with maximal ideal $\p_{F}$, $\mathrm{Pic}(\mathcal{O}_{\p_{F}^{n}})=\{1\}$ for all $n$, so the above bijection becomes 
\begin{equation}
K^\times \backslash \cL \simeq \mathbb{N}, \qquad K^\times \Lambda \mapsto n(\Lambda),  
\end{equation}
where $n(\Lambda)$ is the positive integer for which $c(\Lambda)=\p_{F}^{n(\Lambda)}$. 

\subsection{The action of $K^\times$ on $\cL \times \cL$}\label{subsec:action}
Assume that $\cO_F$ is a discrete valuation ring. We will describe the orbits of $K^{\times}$ acting diagonaly on $\mathcal{L} \times \mathcal{L}$. 

\paragraph{A $K^\times$-invariant of $\cL^2$.} There is an invariant $K^{\times}\backslash\mathcal{L}^2 \twoheadrightarrow (K^{\times}\backslash\mathcal{L})^2 \simeq\mathbb{N}^{2}$,
given by 
\begin{equation}
x=(\Lambda',\Lambda'')\mapsto\underline{n}(x)=\left(n(\Lambda'),n(\Lambda'')\right).\label{eq:DefFineCond}\end{equation}
There is another invariant $K^{\times}\backslash\mathcal{L}^2 \twoheadrightarrow \GL_{F}(V)\backslash\mathcal{L}^{2}\simeq\mathfrak{S}_{2}\backslash\mathbb{Z}^{2}$
that describes the relative position of two lattices. It maps $x=(\Lambda', \Lambda'')$
to the unique pair of integers $\mathrm{inv}(x)=\{i_1, i_2\}\in\mathfrak{S}_{2}\backslash\mathbb{Z}^{2}$
for which there exists an $F$-basis $(e_{1},e_{2})$ of $V$ such that 
\begin{equation}
\Lambda'=\mathcal{O}_{F}e_{1}\oplus\mathcal{O}_{F}e_{2}\quad\mbox{and}\quad \Lambda''=\p_{F}^{i_{1}}e_{1}\oplus \p_{F}^{i_{2}}e_{2}.\label{eq:DefInvRel}
\end{equation}

\paragraph{The distance function.} This latter invariant is related to the distance function on the set of vertices $\cV = F^\times \backslash \cL$ of 
the Bruhat-Tits tree of $\PGL_F (V )$: if $v'$ and $v''$ are the images of $\Lambda'$ and $\Lambda''$ in $\cV$, then 
$\dist(v', v'') = \left | i_1 - i_2\right |$. Since $K^\times \backslash \cL \isom \NN$, also $K^\times \backslash \cV \isom \NN$. 
In other words, the function $n$ on $\cL$ descends to a function $n \colon \cV \ra \NN$ whose fibers 
$\cV(k) = \{v  \in \cV \colon n(v) = k\}$ are precisely the
$K^\times$-orbits in $\cV$. We claim that also $\cV(k) = \{v \in \cV \colon \dist(v, \cV(0)) = k\}$ for all $k \in \NN$.
To prove this, first note that $\cV(0)$ is a convex subset of $\cV$, namely a single vertex
if $K$ is an unramified extension of $F$, a pair of adjacent vertices if $K$ is a ramified
extension of $F$, and a line in the building (i.e. an apartement) if $K \isom F \times F$. 

Now, let $v$ be any vertex in $\cV(k)$. Then $v$ is represented by $\Lambda = \cO_{\p_F^k}  e$ for some $K$-basis 
$e$ of $V$. 
For $0 \leq j \leq k$, let $v_j \in \cV$ be the $F^\times$-homothety class of $\Lambda_j = \cO_{\p_F^j} e \in  \cL$. Then $v_j \in \cV(j)$ and $(v_k, v_{k-1}, \dots , v_0)$ is a path of length $k$ from $v = v_k$ to the vertex $v_0$
of $\cV(0)$. Therefore $\dist(v, \cV(0)) = j$ for some $0 \leq j \leq k$, and the convexity of $\cV(0)$
implies that the $j$th term in our path, namely $v_{k-j}$, should already be in $\cV(0)$.
Therefore $k-j = 0$ and $\dist(v, \cV(0)) = k$.

\paragraph{Counting the $K^\times$-orbits.}
Finally, fix $(n',n'')\in\mathbb{N} \times \mathbb{N}$, $(i_{1},i_{2})\in\Z \times \Z$
and let $\delta=\left|i_{1}-i_{2}\right|$. It follows from the above
considerations that the projection $\mathcal{L}\twoheadrightarrow\mathcal{V}$
induces a bijection between
$$
\mathcal{L}(n',n'';i_{1},i_{2})=\left\{ x\in K^{\times}\backslash\mathcal{L}^{2}\mbox{ s.t. }\underline{n}(x)=(n',n'')\mbox{ and }\mathrm{inv}(x)=\{i_{1},i_{2}\}\right\} 
$$
and the set of $K^{\times}$-orbits of pairs $(v',v'')\in\cV \times \cV$
such that \[
\mathrm{dist}(v',\mathcal{V}(0))=n',\quad\mathrm{dist}(v'',\mathcal{V}(0))=n''\quad\mbox{and}\quad\mathrm{dist}(v',v'')=\delta.\]
If for instance $n'\geq n''$, the choice of a vertex $v'\in\mathcal{V}(n')$
identifies the latter set of $K^\times$-orbits with the set of all vertices $v''\in\mathcal{V}(n'')$
at distance $\delta$ from $v'$. Using then the above description of $\cV(0)$, it is a simple combinatorial exercise to 
prove the following: 

\begin{lem}\label{lem:LocComp}
If $\mathcal{O}_{F}/\p_{F}$ is finite of order
$q$, then $\mathcal{L}(n',n'';i_{1},i_{2})$ is finite of order
\[
N(n',n'',\delta)=\left|\mathcal{L}(n',n'';i_{1},i_{2})\right|\quad\mbox{with }\delta=\left|i_{1}-i_{2}\right|.\]
Moreover $N(n',n'',\delta)=0$ unless one of the following conditions
holds:
\begin{enumerate}
\item $\delta=\left|n'-n''\right|+2r$ for some $0\leq r<\min(n',n'')$.
Then\[
N(n',n'',\delta)=\begin{cases}
1 & \mbox{if }r=0,\\
(q-1)q^{r-1} & \mbox{if }r>0\end{cases}\]

\item $K$ is an unramified extension of $F$ and $\delta=n'+n''$.
Then\[
N(n',n'',\delta)=q^{\min(n',n'')}.\]

\item $K$ is a ramified extension of $F$ and $\delta=n'+n''+s$ with
$s\in\{0,1\}$. Then\[
N(n',n'',\delta)=\begin{cases}
q^{\min(n',n'')} & \mbox{if }s=1\mbox{ or }\min(n',n'')=0,\\
(q-1)q^{\min(n',n'')-1} & \mbox{if }s=0<\min(n',n'').\end{cases}\]

\item $K\simeq F\times F$ and $\delta=n'+n''+s$ with $s\in\mathbb{N}$.
Then \[
N(n',n'',\delta)=\begin{cases}
1 & \mbox{if }\min(n',n'')=0=s,\\
2 & \mbox{if }\min(n',n'')=0<s,\\
(q-2)q^{\min(n',n'')-1} & \mbox{if }\min(n',n'')>0=s,\\
2(q-1)q^{\min(n',n'')-1} & \mbox{if }\min(n',n'',s)>0.\end{cases}\]
\end{enumerate}
\end{lem}
%
%
\section{Adelic constructions of the lifting and the reduction maps}\label{sec:lifting}

\subsection{Notation and preliminaries}\label{subsec:not}
We now switch back to the notation in the introduction, where $F$ is a totally real number field, $B$ is a quaternion algebra over $F$, $K$ is a totally imaginary quadratic extension of $F$ and $\iota \colon K \hra B$ is an embedding. We denote by $\mathrm{Ram}_{f}B$ the set of all finite places of $F$ where $B$ does not split. Let $S$ be a finite set of finite places of $F$ that satisfy the following hypotheses: 

{\bf H.1.} $B$ is unramified at every $v \in S$, i.e.,  $\mathrm{Ram}_{f}B\cap S=\emptyset$, 

{\bf H.2.} $|S| + |\Ram_f(B)|$ + $[F:\Q]$ is even,

{\bf H.3.} Every $v \in S$ is either inert or ramified in $K$.    

Hypotheses {H.1.} and {H.2.} imply that there exists a unique (up to isomorphism) totally 
definite quaternion algebra $B_S$ over $F$ ramified exactly at the finite places $\Ram_f(B) \cup S$. Hypothesis {H.3.} implies that $K$ embeds into $B_S$. 
We fix once and for all such an embedding $\iota_{S}:K\hookrightarrow B_{S}$.
We let 
\[
G=\mathrm{Res}_{F/\mathbb{Q}}B^{\times},\qquad G_{S}=\mathrm{Res}_{F/\mathbb{Q}}B_{S}^{\times}\qquad\mbox{and}\qquad T=\mathrm{Res}_{F/\mathbb{Q}}K^{\times}.
\]

\subsection{Adelic construction of the map $\theta_S$}\label{subsec:adelic}
Following Cornut and Vatsal (see~\cite[\S 2.1]{cornut-vatsal}), we consider the following locally compact and totally discontinous groups:
\[
\left\{ \begin{array}{rcccl}
T(\mathbb{A}_{f}) & = & (K\otimes\mathbb{A}_{f})^{\times} & = & {\textstyle \prod'}K_{v}^{\times}\\
G(\mathbb{A}_{f}) & = & (B\otimes\mathbb{A}_{f})^{\times} & = & {\textstyle \prod'}B_{v}^{\times}\\
G_{S}(\mathbb{A}_{f}) & = & (B_{S}\otimes\mathbb{A}_{f})^{\times} & = & {\textstyle \prod'}B_{S,v}^{\times}\end{array}\right.\enskip\mbox{and}\quad G(S)=\left({\textstyle \prod'_{v\notin S}}B_{S,v}^{\times}\right)\times{\textstyle \prod}_{v\in S}F_{v}^{\times}.
\] 
\noindent The restricted products on the left are the usual ones defined by the arbitrary choice of an integral structure on the 
algebraic groups, whereas the restricted product on the right is induced by that of $G_{S}(\mathbb{A}_{f})$.
These topological groups together fit in a commutative diagram
$$
\xymatrix{
& G(\Af) \ar@{->}[rd]^{\phi_S}&  \\ 
T(\Af) \ar@{->}[ur]^{\iota} \ar@{->}[rd]^{\iota_S} & &  G(S) \\ 
& G_S(\Af) \ar@{->}[ru]^{\pi_S}  &       
} 
$$
where $\iota$ and $\iota_{S}$ are the continuous embeddings induced
by their algebraic counterparts, where $\pi_{S}$ is the continuous,
open and surjective morphism 
\[
G_{S}(\mathbb{A}_{f})={\textstyle \prod}'_{v\notin S}B_{S,v}^{\times}\times {\textstyle \prod}_{v\in S}B_{S,v}^{\times}\twoheadrightarrow {\textstyle \prod}'_{v\notin S}B_{S,v}^{\times}\times{\textstyle \prod}_{v\in S}F_{v}^{\times}=G(S)
\]
which is induced by the identity on $B_{S,v}$ for $v\not\in S$ and by the reduced norm 
$\mathrm{nr}_{S, v} \colon B_{S,v}^{\times}\rightarrow F_{v}^{\times}$ for $v\in S$. Finally, the continuous, open and 
surjective morphism
\[
\phi_S \colon G(\mathbb{A}_{f})={\textstyle \prod}'_{v\notin S}B_{v}^{\times}\times {\textstyle \prod}_{v\in S}B_{v}^{\times}\twoheadrightarrow{\textstyle \prod}'_{v\notin S}B_{S,v}^{\times} \times {\textstyle \prod}_{v\in S}F_{v}^{\times}=G(S)
\]
is again induced by the reduced norm $\mathrm{nr}_{v} \colon B_{v}^{\times}\rightarrow F_{v}^{\times}$
for $v\in S$, and by well-chosen isomorphisms $\theta_{S,v} \colon B_{v}\stackrel{\simeq}{\longrightarrow}B_{S, v}$
for $v\notin S$ whose choice is carefully explained in \cite[\S 2.1.3]{cornut-vatsal}.

\paragraph{Topological properties and Galois action.} We use the above commutative diagram to let 
$T(\mathbb{A}_{f})$ act on $G(\mathbb{A}_{f})$, $G_{S}(\mathbb{A}_{f})$ and $G(S)$ by multiplication on the left.
Let $\overline{T(\mathbb{Q})}$ be the closure of $T(\mathbb{Q})$
in $T(\mathbb{A}_{f})$, and recall that class field theory yields
an isomorphism of topological groups 
$$
\mathrm{Art}_K \colon T(\mathbb{A}_{f})/\overline{T(\mathbb{Q})} \stackrel{\simeq}{\longrightarrow} \Gal_{K}^{ab}.
$$
We thus obtain continuous $\Gal_{K}^{ab}$-equivariant maps of topological
$\Gal_{K}^{ab}$-sets \[
\overline{T(\mathbb{Q})}\backslash G(\mathbb{A}_{f})\stackrel{\phi_{S}}{\longrightarrow}\overline{T(\mathbb{Q})}\backslash G(S)\stackrel{\pi_{S}}{\longleftarrow}\overline{T(\mathbb{Q})}\backslash G_{S}(\mathbb{A}_{f}).\]
These maps are also respectively equivariant for the right actions
of $G(\mathbb{A}_{f})$ and $G_{S}(\mathbb{A}_{f})$. For a compact
open subgroup $H$ of $G(\mathbb{A}_{f})$, define $H(S)=\phi_{S}(H)$
and $H_{S}=\pi_{S}^{-1}(H(S))$. We now have $\Gal_{K}^{ab}$-equivariant
maps of discrete $\Gal_{K}^{ab}$-sets \[
\overline{T(\mathbb{Q})}\backslash G(\mathbb{A}_{f})/H\stackrel{\phi_{S}}{\longrightarrow}\overline{T(\mathbb{Q})}\backslash G(S)/H(S)\stackrel{\pi_{S}}{\longleftarrow}\overline{T(\mathbb{Q})}\backslash G_{S}(\mathbb{A}_{f})/H_{S}\]
and the above $\pi_{S}$ is a bijection by construction of $H_{S}$.
Since \[
\mathcal{CM}(G,H)=T(\mathbb{Q})\backslash G(\mathbb{A}_{f})/H=\overline{T(\mathbb{Q})}\backslash G(\mathbb{A}_{f})/H\]
and similarly for $G_{S}$, we obtain a $\Gal_{K}^{ab}$-equivariant
map of discrete $\Gal_{K}^{ab}$-sets \[
\theta_{S}=\pi_{S}^{-1}\circ\phi_{S}:\mathcal{CM}(G,H)\rightarrow\mathcal{CM}(G_{S},H_{S}).\]
If $H=\prod_{v}H_{v}$ in $G(\mathbb{A}_{f})=\prod'_{v}B_{v}^{\times}$,
then $H_{S}=\prod_{v}H_{S,v}$ in $G_{S}(\mathbb{A}_{f})=\prod'_{v}B_{S,v}^{\times}$
with\[
H_{S,v}=\phi_{v}(H_{v})\mbox{ for }v\notin S\quad\mbox{and}\quad H_{S,v}=\nr_{S,v}^{-1}\left(\nr_{v}(H_{v})\right)\mbox{ for }v\in S.\]
In this case, the map induced by $\theta_{S}$ on the $\Gal_{K}^{ab}$-orbit
spaces, 
\begin{equation}\label{eq:thetabar}
\overline{\theta}_{S} \colon \Gal_{K}^{ab}\backslash\mathcal{CM}(G,H)\rightarrow\Gal_{K}^{ab}\backslash\mathcal{CM}(G_{S},H_{S})
\end{equation}
has a purely local description, namely 
\begin{equation}\label{eq:thetabar-local}
\overline{\theta}_{S}=(\overline{\theta}_{S,v}) \colon {\textstyle \prod'_{v}}K_{v}^{\times}\backslash B_{v}^{\times}/H_{v}\rightarrow{\textstyle \prod'_{v}}K_{v}^{\times}\backslash B_{S,v}^{\times}/H_{S,v}
\end{equation}
where $\overline{\theta}_{S,v}$ is the bijection induced by $\theta_{S, v} \colon B_{v} \stackrel{\simeq}{\longrightarrow} B_{S, v}$ for $v\not\in S$, and equals \[
K_{v}^{\times}\backslash B_{v}^{\times}/H_{v}\stackrel{\nr_{v}}{\longrightarrow}\nr_{v}(K_{v}^{\times})\backslash F_{v}^{\times}/\nr_{v}(H_{v})\stackrel{\nr_{S,v}^{-1}}{\longrightarrow}K_{v}^{\times}\backslash B_{S,v}^{\times}/H_{S,v}\]
for $v\in S$. 

\noindent The construction of the map $\theta_S$ also gives us a reduction map 
\begin{equation}\label{eq:redmap}
\red_S \colon \mathcal{CM}(G, H) \ra \cX(G_S, H_S)
\end{equation}
defined by $\red_S := \pi \circ \theta_S$, where $\pi \colon \mathcal{CM}(G_S, H_S) \ra \cX(G_S, H_S)$ is the natural projection map. 

\subsection{Fine and coarse conductors}\label{subsec:cond}
We now specialize the above constructions to the case $H=\widehat{R}^{\times}$ for some Eichler 
$\mathcal{O}_{F}$-order $R$ in $B$ that will be fixed throughout. The level of $R$ is a nonzero integral ideal of 
$\mathcal{O}_{F}$ that we denote by $\mathfrak{n}$. We also choose two maximal $\mathcal{O}_{F}$-orders
$R'$ and $R''$ in $B$ such that $R=R'\cap R''$. Then $H_{S}=\widehat{R}_{S}^{\times}$
where $R_{S}$ is an Eichler $\mathcal{O}_{F}$-order in $B_{S}$
whose level $\mathfrak{n}_{S}$ is the prime-to-$S$-part of $\mathfrak{n}$. 
We also obtain two maximal $\mathcal{O}_{F}$-orders $R'_{S}$
and $R''_{S}$ in $B_{S}$ such that $R_{S}=R'_{S}\cap R''_{S}$. For a finite place $v$ of $F$, $R_{S, v} = \theta_{S, v}(R_v)$ if $v \notin S$ while $R_{S, v}$ is the unique maximal order of $B_{S, v}$ if $v \in S$.
Explicitly,
\[
R_{S}=\left\{ b\in B_{S} \colon \forall v\notin S,\,\theta_{S,v}^{-1}(b)\in R_{v} \text{ and }\forall v \in S, \ \nr_{S, v}(b) \in \cO_{F,v} \right\}.  
\]
We have a similar explicit description for $R'_{S}$ and $R''_{S}$.

\paragraph{Definitions.} For any finite set of places $T$ of $F$, let $\mathcal{I}(T)$ be the monoid of integral ideals of $\cO_F$ that are coprime to the places of $T$. The \emph{fine conductor} is a $\Gal_K^{\ab}$-invariant map   
\begin{equation}
\bc_{f} \colon \mathcal{CM}(G, H) \ra \mathcal{I}(\Ram_f B) \times \mathcal{I}(\Ram_f B)
\end{equation}
that is defined as follows: given a CM point $x = [g]$ in $\mathcal{CM}(G, H) = T(\Q) \backslash G(\Af) / H$, consider the images $x'$ and $x''$ of $x$ in $T(\Q) \backslash G(\Af) / \widehat{R'}^\times$ and $T(\Q) \backslash G(\Af) / \widehat{R''}^\times$, respectively. 
The intersection $K \cap g\widehat{R'} g^{-1}$ is an $\cO_F$-order in $K$ that depends only on the $\Gal_K^{\ab}$-orbit 
of $x$ and whose conductor $\bc(x')$ is an $\cO_F$-ideal that is prime to $\Ram_f B$. 
Similarly, we obtain an integral ideal $\bc(x'') \subset \cO_F$ for $R''$. One then defines $\bc_{f}(x) := (\bc(x'), \bc(x''))$.  
The \emph{coarse conductor map}   
\begin{equation}
\bc_{g} \colon \mathcal{CM}(G, H) \ra \mathcal{I}(\Ram_f B)
\end{equation}
is defined as $\bc_{g}(x) = \bc(x') \cap \bc(x'') \subset \cO_F$. The stabilizer of $x$ in $T(\Af) = \widehat{K}^\times$ is then 
$$
\Stab_{T(\Af)}(x) = K^\times \cdot \left ( \widehat{K}^\times \cap g \widehat{R}^\times g^{-1} \right ) = K^\times \widehat{\cO_{\bc_{g}(x)}}^\times,  
$$
and the field $K(x)$ that is fixed by the stabilizer of $x$ in $\Gal_K^{\ab}$ is the ring class field $K[\bc_{g}(x)]$ of conductor $\bc_{g}(x)$, i.e., the abelian extension of $K$ fixed by $\rec_K(K^\times \widehat{\cO_{\bc_{g}(x)}}^\times)$. 

Similarly, we define the fine and coarse conductors 
$$
\bc_{S, \f} \colon \mathcal{CM}(G_S, H_S) \ra \cI(\Ram_{f}B_S)^2 \qquad \text{and} \qquad \bc_{S, \g} \colon 
\mathcal{CM}(G_S, H_S) \ra \cI(\Ram_{f} B_S). 
$$
Note that if $y = \theta_S(x)$ then $\bc(y)$ is the prime-to-$S$ part of $\bc(x)$.  We thus have a commutative diagram
\begin{equation}\label{eq:DiagThetaEichler}
\xymatrix{
\mathcal{CM}(G,H) \ar@{->>}[r] \ar@{->}[d]^{\theta_S} & \Gal_{K}^{ab}\backslash\mathcal{CM}(G,H) \ar@{->}[r]^{\mathfrak c_\f} \ar@{->}[d]^{\overline{\theta}_S} & \mathcal{I}(\mathrm{Ram}_{f}B)^{2} \ar@{->}[d]^{()_S} \\
\mathcal{CM}(G_{S},H_{S}) \ar@{->>}[r] & \Gal_{K}^{ab}\backslash\mathcal{CM}(G_{S},H_{S}) \ar@{->}[r]^{\mathfrak c_{S, \f}} & \mathcal{I}(\mathrm{Ram}_{f}B_{S})^{2}, 
}
\end{equation}
where the map $()_{S}$ sends $I \in\mathcal{I}(\Ram_f B)$ to its prime-to-$S$ part $I_{S}\in\mathcal{I}(\Ram_f B_S)$. 

\paragraph{Local analysis.} The right-hand square of this diagram can be analyzed by purely local means: it is the restricted product over all finite primes $v$ of $F$ of  one of the following diagrams:  

\vspace{0.1in}

{\bf Case 1:} $v \notin \Ram_f B \cup S$. 
\[
\xymatrix{
K_{v}^{\times}\backslash B_{v}^{\times}/R_{v}^{\times} \ar@{->}[r]^-{\underline{n}_{v}} \ar@{->}[d]^{\theta_{S, v}} & \mathbb{N} \times \mathbb{N} \ar@{->}[d]^{\mathrm{id}} \\
K_{v}^{\times}\backslash B_{S,v}^{\times}/R_{S,v}^{\times} \ar@{->}[r]^-{\underline{n}_{S,v}} & \mathbb{N} \times \mathbb{N}
}
\]

\vspace{0.1in}

{\bf Case 2:} $v \in S$. 
\[
\xymatrix{
K_{v}^{\times}\backslash B_{v}^{\times}/R_{v}^{\times} \ar@{->}[r]^-{\underline{n}_{v}} \ar@{->}[d]^{\theta_{S, v}} & \mathbb{N} \times \mathbb{N} \ar@{->}[d]^{0} \\
K_{v}^{\times}\backslash B_{S,v}^{\times}/R_{S,v}^{\times} \ar@{->}[r] & 0
}
\]

\vspace{0.1in}

{\bf Case 3:} $v \in \Ram_f B$. 
\[
\xymatrix{
K_{v}^{\times}\backslash B_{v}^{\times}/R_{v}^{\times} \ar@{->}[r]^-{\underline{n}_{v}} \ar@{->}[d]^{\theta_{S, v}} & 0 \ar@{->}[d]^{0} \\
K_{v}^{\times}\backslash B_{S,v}^{\times}/R_{S,v}^{\times}  \ar@{->}[r] & 0
}
\]
 
Here, for $\star \in \{\emptyset,S\}$, the map $\underline{n}_{\star,v}$ sends $K_{v}^{\times}g_{v}R_{\star,v}^{\times}$ to the pair of integers $(n',n'')$ such that $K_{v}\cap g_{v}R'_{\star, v}g_{v}^{-1}$ (resp. $K_{v}\cap g_{v}R''_{\star, v}g_{v}^{-1}$)
is the order of conductor $\mathfrak p_{v}^{n'}$ (resp. $\mathfrak p_v^{n''}$) in $\mathcal{O}_{K_{v}}$, where $\mathfrak p_{v}$ is the maximal ideal in $\mathcal{O}_{F_{v}}$. 

These maps are related to the one defined in \eqref{eq:DefFineCond} as follows: fix a simple left $B_{\star, v}$-module $V_{\star, v}$, let $\mathcal{L}_{v}$ be the set of all $\mathcal{O}_{F_{v}}$-lattices in $V_{\star, v}$ and fix two lattices $\Lambda'_v$ and $\Lambda''_v$ in $\mathcal{L}_{v}$ that are fixed by $(R'_{\star, v})^\times$ and $(R''_{\star, v})^{\times}$, respectively. 
Thus, if 
$(i_1, i_2) = \mathrm{inv}(\Lambda'_v, \Lambda''_v)$ as defined in (\ref{eq:DefInvRel}) then we have 
$\left| i_1 - i_2 \right|=v(\mathfrak{n})$, $R_{\star, v}^{\times}$ is the stabilizer of $(\Lambda'_v, \Lambda''_v) \in \mathcal{L}_{v} \times \mathcal{L}_v$ and $B_{\star, v}^{\times}$ acts transitively on the set $\mathcal{L}_{v}(i_1, i_2)$ of pairs of lattices $(\Lambda', \Lambda'')$ with $\mathrm{inv}(\Lambda', \Lambda'')=(i_1, i_2)$.
Therefore, 
\begin{equation}\label{eq:ident}
K_{v}^{\times}\backslash B_{\star, v}^{\times}/R_{\star, v}^{\times}\simeq K_{v}^{\times}\backslash\mathcal{L}_{v}(i_1, i_2)
\end{equation}
is contained in $K_{v}^{\times}\backslash\mathcal{L}_{v} \times \mathcal{L}_v$. The
map $\underline{n}_{v}$ that we have defined above on the former set is equal to the restriction of the map 
defined by \eqref{eq:DefFineCond} on the latter. In particular, 
the fiber of $\underline{n}_{v}$ over any $(n', n'') \in \mathbb{N} \times \mathbb{N}$
is finite of order $N_{v}(n', n'' ,v(\mathfrak{n}))$ as defined in Lemma~\ref{lem:LocComp} with $q = N(v)$ the order of the residue field of $F$ at $v$.

\subsection{A sign invariant}
In the three diagrams above, the vertical maps are surjective. In
Cases $1$ and $3$, the restriction of the first vertical map to compatible
fibers of the two horizontal maps is still obviously surjective. But
in Case $2$, it may be that such a restriction fails to be surjective.
Since for $v \in S$, 
\[
K_{v}^{\times}\backslash B_{S,v}^{\times}/R_{S,v}^{\times}\simeq
\begin{cases}
\mathbb{Z}/2\mathbb{Z} & \mbox{if }K_{v}/F_{v}\mbox{ is unramified}\\
\{0\} & \mbox{otherwise,}
\end{cases}
\]
this occurs precisely when $v\in S$ is inert in $K$. Let thus $S'$
be the set of all places $v$ in $S$ that are inert in $K$. For
any such $v \in S'$, the valuation of the reduced norm induces a projection  
$$
K_v^\times \backslash B_v^\times / R_v^{\times} \twoheadrightarrow NK_v^\times \backslash F_v^\times / \nr(R_v^\times) \isom\Z/2\Z,  
$$
as well as a bijection $K_v^\times \backslash B_{S, v}^\times / R_{S, v}^\times \isom \Z / 2\Z$. We thus obtain maps
$$
\vphi_{S'} \colon \mathcal{CM}(G, H) \ra \Gal_K^{\ab} \backslash \mathcal{CM}(G, H) \ra 
\prod_{v \in S'} K_v^\times \backslash B_{v}^\times / R_v^\times \ra (\Z / 2\Z)^{|S'|} 
$$  
and 
$$
\psi_{S'} \colon \mathcal{CM}(G_S, H_S) \ra \Gal_K^{\ab} \backslash \mathcal{CM}(G_S, H_S) \twoheadrightarrow  
\prod_{v \in S'} K_v^\times \backslash B_{S, v}^\times / R_{S, v}^\times \isom (\Z / 2\Z)^{|S'|}. 
$$
By construction, $\vphi_{S'} = \psi_{S'} \circ \theta_S$. 


The map $\vphi_{S'}$ admits a description in terms of lattices.  
Under the identification \eqref{eq:ident}, the element $K_v^\times g_v R_v^\times$ corresponds to the pair of lattices 
$(g_v \Lambda_v', g_v \Lambda_v'')$. Let $y'$ and $y''$ be the images of $\Lambda_v'$ and $\Lambda_v''$ in the Bruhat--Tits tree $\cV_v = F_v^\times \backslash \cL_v$ and recall from Section~\ref{sec:ql} that for all $k \in \mathbb{N}$, 
$$
\cV_v(k) = \{y \in \cV_v \colon n(y) = k\} = \{y \in \cV_v \colon \dist(y, y_0) = k\}, 
$$
where $\cV_v(0) = \{y_0\}$ since $v$ is inert in $K$. Therefore, 
$$
n(g_v y') - n(y') = \dist(g_v y', y_0) - \dist(y', y_0) \equiv \dist(g_v y', y') \equiv v(\det(g_v)) \textrm{ mod } 2
$$
and similarly for $y''$. Note that the same argument also shows that 
$$
n(y') - n(y'') \equiv v(\mathfrak n) \textrm { mod } 2. 
$$
The map $\vphi_{S'}$ can thus be computed as follows: for $x \in \mathcal{CM}(G, H)$ with $\bc_{f}(x) = (\mathfrak c' , \mathfrak c'')$, 
\begin{equation}\label{eq:defphiS}
\vphi_{S'}(x) = (v(\mathfrak c') - n(\Lambda_v') \textrm{ mod }2)_{v \in S'} = (v(\mathfrak c'') - n(\Lambda_v'') \textrm{ mod }2)_{v \in S'} \in (\Z / 2\Z)^{|S'|}. 
\end{equation}
In particular, $\vphi_{S'}$ factors through the fine conductor map $\bc_{f}$. 

\subsection{Statement of the main theorem}\label{subsec:corresp}





Fix $(\bc', \bc'') \in \cI(\Ram_{f} B)$ such that $v(\mathfrak c') - v(\mathfrak c'') \equiv v(\n)$ mod 2 for every $v \in S'$.  Let 
$e_{S'}(\mathfrak c', \mathfrak c'')$ be the element of $(\Z/2\Z)^{|S'|}$ defined by the right-hand side of \eqref{eq:defphiS}. 
Let $\mathfrak c = \mathfrak c' \cap \mathfrak c''$ and let 
$$
\kappa = |\Gal(K[\mathfrak c] / K[\mathfrak c_S])| \prod_{v \in S}N_v(v(\bc'), v(\bc''), v(\n)),  
$$ 
where $N_v(n', n'', \delta)$ is the function defined in Lemma~\ref{lem:LocComp} with $q = N(v)$ being the order of the residue field of $F$ at $v$.  


\begin{thm}\label{thm:main}
By restriction, the map $\theta_S$ induces a surjective $\kappa$-to-1 correspondence
$$
\theta_S \colon \bc_{f}^{-1}(\mathfrak c', \mathfrak c'') \ra \bc_{S, \f}^{-1}(\mathfrak c'_S, \mathfrak c''_S) \cap \psi_{S'}^{-1}(e_{S'}(\mathfrak c', \mathfrak c'')). 
$$
\end{thm}

\noindent By restricting to a subspace of the target space, we immediately obtain the following:

\begin{cor}\label{cor:main}
For any point $s \in \cX(G_S, H_S)$, the map $\theta_S$ induces a $\kappa$-to-1 correspondence 
\begin{equation}\label{eq:corr}
\theta_S \colon \red_S(s)^{-1} \cap \bc_{f}^{-1}(\mathfrak c', \mathfrak c'') \ra \pi^{-1}(s) \cap \bc_{S, \f}^{-1}(\mathfrak c'_S, \mathfrak c''_S) \cap \psi_{S'}^{-1}(e_{S'}(\mathfrak c', \mathfrak c'')). 
\end{equation}
\end{cor}

\subsection{Computation of the fiber for the definite algebra $B_S$}

Before proving the theorem, we shall explain how to compute the right-hand side of the correspondence \eqref{eq:corr}. Fix $s \in \cX(G_S, H_S)$ and some $g \in G_S(\Af)$ above $s$. Then $b \mapsto b^{-1}g$ induces a bijection 
$$
R_{S, g}^\times \backslash G_S(\Q) / T(\Q) \stackrel{\sim}{\longrightarrow} \pi^{-1}(s), 
$$
where $R_{S, g} = g \widehat{R_S}g^{-1} \cap B_S$ is an Eichler order of level $\n_S$ in $B_S$. On the other hand, the map $b \mapsto \text{ad}(b) \circ \iota_S$ induces a bijection 
$$
G_S(\Q ) / T(\Q) \stackrel{\sim}{\longrightarrow} \Hom_{F-\alg}(K, B_S) 
$$
that is $G_S(\Q) = B_S^\times$-equivariant for the natural left actions on both sides. Combining these two identifications, we obtain 
$$
\pi^{-1}(s) \isom R_{S, g}^\times \backslash \Hom_{F-\alg}(K, B_S). 
$$
Let also $R_{S, g}' = g \widehat{R_S'}g^{-1} \cap B_S$ and $R_{S, g}'' = g \widehat{R_S''}g^{-1} \cap B_S$. These are maximal orders in $B_S$ and $R_{S, g} = R_{S,g}' \cap R_{S, g}''$. Under these identifications, 
\begin{itemize}
\item The restriction to $\pi^{-1}(s)$ of the fine conductor map $\bc_{S, f} \colon \mathcal{CM}(G_S, H_S) \ra \cI(\Ram_{f}B_S)^2$ is induced by the $R_{S, g}^\times$-invariant map 
$$
\bc_{S, f}^g \colon \Hom_{F-\alg}(K, B_S) \ra \cI(\Ram_{f} B_S)^2
$$
whose fiber over $(\bc_S', \bc_S'')$ consists of the embeddings $j \colon K \ra B_S$ such that 
$j^{-1}(R_{S, g}') = \cO_{\bc'_S}$ and $j^{-1}(R_{S, g}'') = \cO_{\bc''_S}$, so that $j^{-1}(R_{S, g}) = \cO_{\bc_S}$ with $\bc_S = \bc'_S \cap \bc''_S$. 

\item The restriction to $\pi^{-1}(s)$ of the sign invariant $\psi_{S'} \colon \mathcal{CM}(G_S, H_S) \ra (\Z / 2\Z)^{|S'|}$ is induced by the $R_{S, g}^\times$-invariant map 
$$
\psi_{S'}^g \colon \Hom_{F-\alg}(K, B_S) \stackrel{\text{sp}}{\longrightarrow} \prod_{v \in S'}\Hom_{\F(v)}(\mathbb{K}(v), \mathbb{B}_S(v)) 
\stackrel{\sim}{\ra} (\Z / 2\Z)^{|S'|}, 
$$
where $\F(v)$, $\mathbb{K}(v)$ and $\mathbb{B}_S(v)$ are the residue fields of the maximal 
$\cO_{F_v}$-orders in $F_v$, $K_v$ and $B_{S, v}$, respectively and $\text{sp}$ is the natural specialization map. Moreover, the last bijection is the unique isomorphism of $(\Z / 2\Z)^{|S'|}$-torsors that maps $\text{sp}(\iota_S)$ to $\psi_{S'}(s) = (v(\det g_v) \text{ mod }2)_{v \in S'}$. 
\end{itemize}

\begin{rem}
For a coarse conductor $\bc \in \cI(\Ram_{f}B)$ that is prime to the level $\n$, we may replace the fine conductor maps by the coarse ones in the above statements since 
$$
\bc_{f}(\bc, \bc) = \bc_g^{-1}(\bc) \qquad \text{and} \qquad \bc_{S, f}^{-1}(\bc_S, \bc_S) = \bc_{S, g}^{-1}(\bc_S). 
$$
In this situation, the right-hand side of Corollary~\ref{cor:main} counts the number of $R_{S, g}^\times$-conjugacy classes of optimal embeddings $\cO_{\bc_S} \hra R_S$ that induce a given collection of isomorphisms $\mathbb{K}(v) \ra \mathbb{B}_S(v)$ between the corresponding residue field at $v \in S'$. If, in addition, $S'$ consists of a single prime $S' = \{\ell\}$, one can remove this last condition by identifying the embeddings that are conjugated by the non-trivial automorphism of $\Gal(K/F) \isom \Gal(\mathbb{K}(v) / \mathbb{F}(v))$. 
\end{rem}


\begin{rem}
For $F = \Q$, one can interpret the right-hand side of \eqref{eq:corr} in terms of primitive representations of integers by ternary quadratic forms (see \cite[pp.172--173]{gross:heights} and 
\cite[Prop.4.2]{jetchev-kane} for details). 
\end{rem}

%
%
\subsection{Proof of Theorem~\ref{thm:main}}\label{sec:proof}

To prove the theorem, let $(\mathfrak{c}_{S}',\mathfrak{c}_{S}'')\in\mathcal{I}(\mathrm{Ram}_{f}B_{S})^{2}$ be the prime-to-$S$ parts of
$(\mathfrak{c}', \mathfrak{c}'') \in \mathcal{I}(\mathrm{Ram}_{f} B)^2$. 
Consider the diagram
$$
\xymatrix{
\bc_{\f}^{-1}(\mathfrak{c}',\mathfrak{c}'') \ar@{->>}[r] \ar@{->}[d]^{\theta_S^{(1)}} & \Gal_{K}^{ab}\backslash \mathfrak{c}_{\f}^{-1}(\mathfrak{c}',\mathfrak{c}'') \ar@{->}[d]^{\theta_S^{(2)}} \\
\mathfrak{c}_{S, \f}^{-1}(\mathfrak{c}_{S}',\mathfrak{c}_{S}'') \ar@{->>}[r] & \Gal_{K}^{ab}\backslash \mathfrak{c}_{S, \f}^{-1}(\mathfrak{c}_{S}',\mathfrak{c}_{S}'')
}
$$
obtained from the first square of \eqref{eq:DiagThetaEichler} by restriction to the relevant fibers of $\mathfrak{c}_{\f}$ and 
$\mathfrak{c}_{S, \f}$. Note that the vertical maps may fail to be surjective, but the local analysis of Section~\ref{subsec:action} shows that 
\begin{lem}\label{lem:locgalorb}
The map $\theta_S^{(2)}$ maps $\Gal_{K}^{ab}\backslash \mathfrak{c}_{\f}^{-1}(\mathfrak{c}',\mathfrak{c}'')$ onto 
$ \Gal_K^{\ab} \backslash \left (\bc_{S, \f}^{-1}(\bc_S', \bc_S'') \cap \psi_{S'}^{-1}(e_{S'}(\bc', \bc''))\right )$ and moreover, 
it is $k^{(2)}$-to-1, where  
$$
k^{(2)}=\prod_{v\in S}N_{v}\left(v(\mathfrak{c}'),v(\mathfrak{c}''),v(\mathfrak{n})\right). 
$$
\end{lem}

\noindent The same local analysis allows us to compute the number of Galois orbits of CM points with a prescribed fine conductor for each of the algebras 
$B$ and $B_S$:  
\begin{lem}\label{lem:galorb}
(i) The number of $\Gal_{K}^{ab}$-orbits in $\mathcal{CM}(G,H)$ with
fine conductor $(\mathfrak{c}',\mathfrak{c}'')\in\mathcal{I}(\mathrm{Ram}_{f}B)^{2}$
is finite and equal to 
$$
N_{B}(\mathfrak{c}',\mathfrak{c}'',\mathfrak{n})=  2^{\#\{v \in \mathrm{Ram}_{f} B,\ v\ \mathrm{inert\ in\ }K\}} \prod_{v\notin\mathrm{Ram}_{f}B}N_{v}\left(v(\mathfrak{c}'),v(\mathfrak{c}''),v(\mathfrak{n})\right), 
$$
where the product is taken over all finite primes. The number of CM points in each of these orbits is equal to 
\[
h(\mathfrak{c})=\left|\Gal(K[\mathfrak{c}]/K)\right|\quad\mbox{where }\mathfrak{c}=\mathfrak{c}'\cap\mathfrak{c}''.
\]

\noindent (ii) The number of $\Gal_{K}^{ab}$-orbits in $\mathcal{CM}(G_S,H_S)$ with fine conductor 
$(\mathfrak{c}'_S,\mathfrak{c}''_S)\in\mathcal{I}(\mathrm{Ram}_{f}B_S)^{2}$ is finite and equal to 
$$
N_{B_S}(\mathfrak{c}'_S,\mathfrak{c}''_S,\mathfrak{n})=  2^{\#\{v \in \mathrm{Ram}_{f} B_S,\ v\ \mathrm{inert\ in\ }K\}} \prod_{v\notin\mathrm{Ram}_{f}B_S}N_{v}\left(v(\mathfrak{c}'_S),v(\mathfrak{c}''_S),v(\mathfrak{n})\right), 
$$
where the product is taken over all finite primes. The number of CM points in each of these orbits is equal to 
\[
h(\mathfrak{c}_S)=\left|\Gal(K[\mathfrak{c}_S]/K)\right|\quad\mbox{where }\mathfrak{c}_S=\mathfrak{c}'_S\cap\mathfrak{c}''_S.
\]
\end{lem}

\noindent Theorem~\ref{thm:main} now follows easily from Lemma~\ref{lem:locgalorb}. 



%
%
\section{Application to Galois orbits of Heegner points on modular curves}\label{sec:apps}

\subsection{The modular curve $Y_0(N)$}
Let $N\geq1$ be an integer and let $S$ be a scheme over $\Spec \Z\left [\frac{1}{N} \right ]$. An
enhanced elliptic curve $(E,C)$ over $S$ is an elliptic curve $E$ over $S$ together with a closed subgroup $C$ that is locally isomorphic to the constant group scheme
$(\mathbb{Z}/N\mathbb{Z})_{S}$ for the \'etale topology on $S$. The modular curve 
$Y=Y_{0}(N)$ is a smooth affine curve over $\Spec \mathbb{Z}\left [\frac{1}{N} \right]$ that coarsely 
represents the contravariant functor mapping $S$ to the set of isomorphism
classes of enhanced elliptic curves over $S$. Here, we say that two enhanced elliptic curves 
$(E_{1},C_{1})$ and $(E_{2},C_{2})$ over $S$ are isomorphic if there is an isomorphism 
$E_{1}\stackrel{\simeq}{\rightarrow} E_{2}$ of elliptic curves over $S$ that maps $C_{1}$ to $C_{2}$. 
We denote by $[E,C]$ the $S$-valued point of $Y$ defined by an enhanced elliptic curve $(E,C)$ over $S$. For $S=\Spec \mathbb{C}$, we have the usual isomorphism
\begin{equation}
\Gamma_{0}(N)\backslash\h \stackrel{\simeq}{\longrightarrow}Y(\mathbb{C})\qquad\tau\mapsto[E_{\tau},C_{\tau}]\label{eq:HisoY}
\end{equation}
where $\h$ is the upper half-plane, $\ds \Gamma_{0}(N)=\left\{ {\mtwo a b c d } \in \SL_{2}(\mathbb{Z}) \colon c\equiv0\bmod N\right\} $ acts on $\h$ by $\ds {\mtwo a b c d} \cdot \tau=\frac{a\tau+b}{c\tau+d}$, and $E_{\tau}=\mathbb{C}/\left\langle 1,\tau\right\rangle $ with
$C_{\tau}=\left\langle N^{-1},\tau\right\rangle /\left\langle 1,\tau\right\rangle $.

\subsection{Isogeny classes}
Fix an elliptic curve $\cE$ over a base $S$. Let $\mathcal{I}(\mathcal{E}_{\star})$ be the contravariant functor that assigns to an $S$-scheme $T$ the set \emph{$\mathcal{I}(\mathcal{E}_{T})$} of isomorphism classes of triples $(E,C,\phi)$ where $(E,C)$ is
an enhanced elliptic curve over $T$ and where $\phi$ is an invertible
element of $\Hom_{T}(E/C,\mathcal{E}_{T})\otimes\mathbb{Q}$. 
An isomorphism between two such triples $(E_{1},C_{1},\phi_{1})$
and $(E_{2},C_{2},\phi_{2})$ is an isomorphism of enhanced elliptic
curves $\theta \colon (E_{1},C_{1})\rightarrow(E_{2},C_{2})$ such that $\phi_{2}\circ\overline{\theta}=\phi_{1}$, where $\overline{\theta} \colon E_{1}/C_{1}\rightarrow E_{2}/C_{2}$ is the
induced isomorphism. The group $\Aut_{T}^{0}(\mathcal{E}_{T})$ of
invertible elements in $\End_{T}^{0}(\mathcal{E}_{T}) = \End_T(\mathcal{E}_T) \otimes \Q$ 
acts on $\mathcal{I}(\mathcal{E}_{T})$ by 
$$
\sigma\cdot(E,C,\phi)=(E,C,\sigma\circ\phi).
$$ 
If $s \in S$ is a geometric point, the map $(E,C,\phi)\mapsto[E,C]$ yields a bijection
\[
\Aut_{s}^{0}(\mathcal{E}_{s})\backslash\mathcal{I}(\mathcal{E}_{s})\simeq Y(\mathcal{E}_{s})=\left\{ x\in Y(s)\vert x=[E,C]\mbox{ s.t. }\Hom_{s}^{0}(E,\mathcal{E}_{s})\neq0\right\} \subset Y(s).
\]

\subsection{Lattices}
For a geometric point $s$ of $S$ and a prime number $p$ let $T_{p}(\mathcal{E}_{s})$ be the $p$-adic Tate module of $\mathcal{E}_{s}$ if $\chr(s)\neq p$ and the covariant Dieudonn\'e crystal
of the $\ell$-divisible group $\mathcal{E}_{s}[\ell^{\infty}]$ if $p=\ell=\chr(s)$.
Let also $\widehat{T}(\mathcal{E}_{s})=\prod_{p}T_{p}(\mathcal{E}_{s})$
and $\widehat{V}(\mathcal{E}_{s})=\widehat{T}(\mathcal{E}_{s})\otimes\mathbb{Q}$.
If $\chr(s)=0$, a lattice in $\widehat{V}(\mathcal{E}_{s})$
is any $\widehat{\mathbb{Z}}$-submodule that is commensurable with
$\widehat{T}(\mathcal{E}_{s})$. If $\chr(s)=\ell$, a lattice
in $\widehat{V}(\mathcal{E}_{s})=\widehat{V}^{(\ell)}(\mathcal{E}_{s})\times V_{\ell}(\mathcal{E}_{s})$ is a submodule of the form $\widehat{T}^{(\ell)}\times T_{\ell}$ where 
$\widehat{T}^{(\ell)}$ is a $\widehat{\mathbb{Z}}^{(\ell)}$-submodule of $\widehat{V}^{(\ell)}(\mathcal{E}_{s})$ that is commensurable with $\widehat{T}^{(\ell)}(\mathcal{E}_{s})$, and
where $T_{\ell}$ is a subcrystal of $V_{\ell}(\mathcal{E}_{s})$. Here, 
$$
\widehat{\mathbb{Z}}^{(\ell)} = \prod_{p \ne \ell} \Z_{p} \qquad \widehat{V}^{(\ell)}(\mathcal{E}_{s}) = \prod_{p \ne \ell} V_p(\mathcal{E}_s) \qquad \ds \widehat{T}^{(\ell)}(\mathcal{E}_{s}) = \prod_{p \ne \ell} T_p(\mathcal{E}_s).
$$ 

Let $\mathcal{L}(\mathcal{E}_{s})$
be the set of pairs of lattices $(\widehat{T}_{1},\widehat{T}_{2})$ in $\widehat{V}(\mathcal{E}_{s})$ such that $\widehat{T}_{1}\subset\widehat{T}_{2}$ and $\widehat{T}_{2}/\widehat{T}_{1}\simeq\mathbb{Z}/N\mathbb{Z}$ (thus if $\chr(s)=\ell$, $\widehat{T}_{1}$ and $\widehat{T}_{2}$ share the same $\ell$-component, since $\ell \nmid N$). For $(E,C,\phi)\in\mathcal{I}(\mathcal{E}_{s})$,
we have morphisms
\[
\widehat{T}(E)\stackrel{\mathrm{can}}{\longrightarrow}\widehat{T}(E/C)\subset\widehat{V}(E/C)\stackrel{\phi}{\longrightarrow}\widehat{V}(\mathcal{E}_{s})
\]
and the resulting map 
\[
(E,C,\phi)\mapsto\left(\phi\circ\mathrm{can}\left(\widehat{T}(E)\right),\phi\left(\widehat{T}(E/C)\right)\right)
\]
yields a bijection $\mathcal{I}(\mathcal{E}_{s})\simeq\mathcal{L}(\mathcal{E},s)$.
Thus also 
\[
Y(\mathcal{E}_{s})\simeq\Aut_{s}^{0}(\mathcal{E}_{s})\backslash\mathcal{L}(\mathcal{E},s).
\]

\subsection{CM points as special points}
Let $K$ be a quadratic imaginary field and let 
$K\hookrightarrow\mathbb{C}$ be a fixed embedding. An elliptic curve $E$ over a 
field $F$ is said to have complex multiplication by $K$ if $\End_{F}^{0}(E)$ 
is isomorphic to $K$. If $F$ is a subfield of $\mathbb{C}$, we
normalize the isomorphism $K\simeq\End_{F}^{0}(E)$ by requiring that
$K$ acts on the tangent space $\Lie E(\mathbb{C})$ through our fixed
embedding $K\hookrightarrow\mathbb{C}$. The conductor of $E$ is
the unique positive integer $c(E)$ such that $\End_{F}(E)\simeq\mathbb{Z}+c(E)\mathcal{O}_{K}$
inside $\End_{F}^{0}(E)\simeq K$. A complex point $x\in Y(\mathbb{C})$
is said to have complex multiplication by $K$ if $x=[E,C]$ for some
elliptic curve $E$ over $\mathbb{C}$ with complex multiplication
by $K$. The elliptic curve $E/C$ then also has complex multiplication by $K$. The
fine and coarse conductors of $x$ are respectively equal to 
\[
\bc_{f}(x)=(c(E),c(E/C))\in\mathbb{N} \times \mathbb{N}\quad\mbox{and}\quad \bc_{g}(x)=\mathrm{lcm}(c(E),c(E/C))\in\mathbb{N}.
\]
We denote by $\mathcal{CM}_{K}$ the subset of $Y(\mathbb{C})$ thus
defined and refer to its elements as CM points. Note that the bijection~(\ref{eq:HisoY}) 
restricts to 
\begin{equation}
\Gamma_{0}(N)\backslash \left ( \h \cap K \right ) \stackrel{\simeq}{\longrightarrow}\mathcal{CM}_{K}\qquad\tau\mapsto[E_{\tau},C_{\tau}].
\label{eq:HKisoCM}
\end{equation}
Let $\tau\in\h\cap K$ satisfy the quadratic equation $A\tau^{2}+B \tau+C = 0$ where $A > 0$, $A, B, C \in \Z$ and $(A, B, C) = 1$.  One can calculate the fine conductor of $[E_\tau, C_\tau]$ as follows: if $\ds \Delta(\tau) = B^2 - 4AC$ is the discriminant modular function and if $D < 0$ is the fundamental discriminant of $K$ then (see, e.g., 
\cite[\S 7]{cox:book})
\[
\bc_{f}[E_{\tau},C_{\tau}]= \left ( \sqrt{\left | \frac{\Delta(\tau)}{D} \right | }, 
\sqrt{ \left | \frac{\Delta(N\tau)}{D} \right | }\right ) .  
\]
The point $[E_\tau, C_\tau]$ is a Heegner point if and only if $\Delta(\tau) = \Delta(N\tau)$. It is not hard to show that the latter is equivalent to $A \mid N$ and $(A/N, B, NC) = 1$. 

\subsection{CM points as isogeny class}
All elliptic curves over $\C$ with complex multiplication
by $K$ are isogenous. Fix one such curve $\mathbb{E}$.
Then (notation as above) 
\begin{equation}
\mathcal{CM}_{K}=Y(\mathbb{E})\stackrel{\simeq}{\longleftarrow}K^{\times}\backslash\mathcal{I}(\mathbb{E})\stackrel{\simeq}{\longrightarrow}K^{\times}\backslash\mathcal{L}(\mathbb{E},\mathbb{C}). 
\label{eq:CMandLat}
\end{equation}
The fine conductor corresponds to 
\[
\bc_{f}(\widehat{T}_{1},\widehat{T}_{2})=\left(c(\widehat{T}_{1}),c(\widehat{T}_{2})\right)
\]
where for a lattice $\widehat{T}$ in $\widehat{V}(\mathbb{E})$, $c(\widehat{T})$ is the conductor of the quadratic order 
$$
\cO_{c(\widehat{T})} = \{s\in K \colon s\widehat{T}\subseteq \widehat{T}\}.
$$ 
The $\cO_{c(\widehat{T})}$-module $\widehat{T}$ is thus free of rank one. 

\subsection{The Galois action on CM points}
Let $K^{ab}$ be the maximal abelian extension of $K$ inside $\mathbb{C}$,
and let $\mathrm{Art}_{K} \colon \widehat{K}^{\times}\rightarrow\Gal(K^{ab}/K)$
be the reciprocal of the usual Artin reciprocity. In other words, $\mathrm{Art}_K$ sends uniformizers to geometric Frobenii \cite[p.90]{milne:cm}. The main theorem of complex multiplication then 
says (see \cite[Thm.3.10]{milne:cm2} or \cite[Thm.9.10]{milne:cm}): 
\begin{quote}
For any $\sigma\in\Gal(\mathbb{C}/K)$ and any $s\in\widehat{K}^{\times}$
such that $\sigma\vert K^{ab}=\mathrm{Art}_{K}(s)$ in $\Gal(K^{ab}/K)$,
there exists a unique isogeny $\lambda \colon \mathbb{E} \ra \sigma\mathbb{E}$
such that for all $y \in\widehat{V}(\mathbb{E})$, $\lambda(sy) = \sigma y$
in $\widehat{V}(\sigma\mathbb{E})$. 
\end{quote}
Suppose now that $x=[E,C]$ belongs to $\mathcal{CM}_{K}$, corresponding
to a pair of lattices $(\widehat{T}_{1},\widehat{T}_{2})$ in $\mathcal{L}(\mathbb{E},\mathbb{C})$
obtained from $(E,C)$ by choosing s non-zero $\phi \colon E/C \ra \mathbb{E}$.
Fix $\sigma$, $s$ and $\lambda$ as above. Then $\sigma x=[\sigma E,\sigma C]$
belongs to $\mathcal{CM}_{K}$, and it corresponds to the pair of
lattices $(s\widehat{T}_{1},s\widehat{T}_{2})$ for the choice 
$$
\phi'=\lambda^{-1}\circ\sigma\phi \colon \sigma E/\sigma C \ra \mathbb{E}.
$$ 
The bijection~(\ref{eq:CMandLat}) therefore maps the action of 
$\sigma\in\Gal(\mathbb{C}/K)$ on $\mathcal{CM}_{K}$ to left multiplication by 
$s\in\widehat{K}^{\times}$ on $\mathcal{L}(\mathbb{E},\mathbb{C})$. It follows that the field of definition of $x \in \mathcal{CM}_{K}$
is equal to the ring class field $K[\bc_{g}(x)]$ specified by the coarse
conductor $\bc_{g}(x)$ of $x$. In fact, it is well-known that the corresponding point 
$x\in Y(K[\bc_{g}(x)])$ can be represented by an enhanced elliptic curve $(E,C)$ over 
$K[\bc_{g}(x)]$.

\subsection{The good reduction of CM points}

Let $\overline{\mathbb{Q}}$ be the algebraic closure of $\mathbb{Q}$
inside $\mathbb{C}$. Let $\ell \nmid N$ be a prime. Fix an embedding 
$\iota_\ell \colon \overline{\mathbb{Q}} \hookrightarrow\overline{\mathbb{Q}}_{\ell}$. Note that $\iota_\ell$ determines a valuation ring $\overline{\mathcal{O}}\subset\overline{\mathbb{Q}}$ whose residue field $\overline{\mathbb{F}}_\ell$ is an algebraic closure of $\mathbb{F}_{\ell}$ and a place $\lambar$ of $\Qbar$ over $\ell$. Let $K[\infty]\subset K^{\ab}$ be the union of all ring class fields $K[c]$. For each $c$ (including $c = \infty$), let $\lambda_c$ be the place of $K[c]$ below $\lambar$. Let $\cO[c]$ be the valuation ring of $\lambda_c$ inside $K[c]$ and let $\F[c]$ be its residue field. 

There are various equivalent approaches to the reduction theory of CM points at $\lambda$:  

\paragraph{1. Reduction via properness.} Let $X_{/\Spec\mathbb{Z}[1/N]}$ be the smooth and proper compactification of $Y$ constructed by Deligne--Rapoport \cite{deligne-rapoport} when $N$ is prime and by Katz--Mazur \cite{katzmazur} in general. 
One gets a reduction map 
$$
\red \colon Y(K[c]) \hra X(K[c]) \simeq X(\cO[c]) \xra{\red_{\lambda_c}} X(\F[c]),  
$$
where the bijection $X(K[c]) \simeq X(\cO[c])$ follows from the valuative criterion of 
properness. 

\paragraph{2. Reduction via N\'eron models.} For a point $x\in\mathcal{CM}_{K}$ 
with coarse conductor $c=\bc_{g}(x)$ write $x=[E,C]$ for some enhanced elliptic curve 
$(E,C)$ over $K[c]$ with complex multiplication by $K$. Suppose first that $E$ has
good reduction at $\lambda_c$. Then $(E,C)$ extends to an enhanced elliptic
curve $(\mathcal{E},\mathcal{C})$ over $\mathcal{O}[c]$ where
$\mathcal{E}_{/\cO[c]}$ is the N\'eron model of $E$. The special fiber of the 
latter gives a point $\mathrm{red}(x)=[\mathcal{E}_{\mathbb{F}[c]},\mathcal{C}_{\mathbb{F}[c]}]$ in $Y(\mathbb{F}[c])$. If $E$ does not have good reduction at 
$\lambda_c$, we know by \cite[Thm.7]{serre-tate} that $E$ 
acquires good reduction at $\lambar$ (and indeed everywhere) after a suitable
cyclic extension of $K[c]$. We may thus define $\mathrm{red}(x)$ as the point corresponding 
to the special fiber of the N\'eron model of the base change of $(E,C)$ to such an
extension. This yields a well-defined reduction map
\[
\mathrm{red} \colon \mathcal{CM}_{K}\rightarrow Y(\overline{\mathbb{F}}_\ell).
\]
The above two constructions give the same map on $\mathcal{CM}_{K}$ with
values in $X(\Fbar_\ell)$. Since $Y(\F[\infty])=X(\mathbb{F}[\infty]) \cap Y(\Fbar_\ell)$
in $X(\overline{\mathbb{F}}_\ell)$, we have a well-defined map
\begin{equation}
\mathrm{red} \colon \mathcal{CM}_{K}\rightarrow Y(\mathbb{F}[\infty]).\label{eq:reducDef1}
\end{equation}
For a more explicit construction of this map, we can also choose an
elliptic curve $\mathbb{E}$ over $K[\infty]$ with complex multiplication
by $K$ and good reduction at $\lambda_\infty$. Such a curve exists by the elementary 
theory of complex multiplication and by \cite[Cor.1]{serre-tate}. We then
reduce $\mathcal{CM}_{K}$ as the isogeny class $Y(\mathbb{E})$,
as we shall now explain. 

\paragraph {3. Reduction maps for isogeny classes.}
Let $S=\Spec\overline{\mathcal{O}}$ with geometric points $g=\Spec \overline{\mathbb{Q}}$
 and $s=\Spec\overline{\mathbb{F}}_\ell$. Let $\mathcal{E}_{/S}$ be an elliptic curve. We will eventually take $\mathcal{E}$
to be the N\'eron model of an elliptic curve $\mathbb{E}$ with complex
multiplication by $K$, but in what follows we will not make this assumption. The theory of N\'eron models implies that 
restriction from $S$ to $g$ yields bijections 
\[
\mathcal{I}(\mathcal{E})\stackrel{\simeq}{\longrightarrow}\mathcal{I}(\mathcal{E}_{g})\qquad\mbox{and}\qquad\End_{S}(\mathcal{E})\stackrel{\simeq}{\longrightarrow}\End_{g}(\mathcal{E}_{g}).
\]
Restriction from $S$ to $s$ on the other hand gives 
\[
\mathcal{I}(\mathcal{E})\rightarrow\mathcal{I}(\mathcal{E}_{s})\qquad\mbox{and}\qquad\End_{S}(\mathcal{E})\hookrightarrow\End_{s}(\mathcal{E}_{s}).
\]
We get reduction maps $\red \colon \cL(\cE, g) \ra \cL(\cE, s)$ (from left to right):  
\[
\begin{array}{ccccc}
\mathcal{L}(\mathcal{E},g) & \stackrel{\simeq}{\longleftarrow} & \mathcal{I}(\mathcal{E}) & \rightarrow & \mathcal{L}(\mathcal{E},s)\\
\uparrow\simeq &  & \parallel &  & \uparrow\simeq\\
\mathcal{I}(\mathcal{E}_{g}) & \stackrel{\simeq}{\longleftarrow} & \mathcal{I}(\mathcal{E}) & \rightarrow & \mathcal{I}(\mathcal{E}_{s})\\
\downarrow &  & \downarrow &  & \downarrow\\
Y(\mathcal{E}_{g})=\Aut_{g}^{0}(\mathcal{E}_{g})\backslash\mathcal{I}(\mathcal{E}_{g}) & \stackrel{\simeq}{\longleftarrow} & \Aut_{S}^{0}(\mathcal{E})\backslash\mathcal{I}(\mathcal{E}) & \rightarrow & \Aut_{s}^{0}(\mathcal{E}_{s})\backslash\mathcal{I}(\mathcal{E}_{s})=Y(\mathcal{E}_{s})\end{array}
\]
The map $\mathrm{red} \colon \mathcal{L}(\mathcal{E},g)\rightarrow\mathcal{L}(\mathcal{E},s)$
is induced by the natural isomorphism $\widehat{T}^{(\ell)}(\mathcal{E}_g)\stackrel{\simeq}{\longrightarrow}\widehat{T}^{(\ell)}(\mathcal{E}_s)$
between the Tate modules away from $\ell$, together with a map
\begin{equation}
\mathrm{red}_{\ell} \colon \mathcal{L}_{\ell}(\mathcal{E}, {g})\rightarrow\mathcal{L}_{\ell}(\mathcal{E},s)
\label{eq:redp}
\end{equation}
from the set $\mathcal{L}_{\ell}(\mathcal{E}, g)$ of $\mathbb{Z}_{\ell}$-lattices
in $V_{\ell}(\mathcal{E}_g)$ to the set $\mathcal{L}_{\ell}(\mathcal{E}, s)$
of crystals in $V_{\ell}(\mathcal{E}_s)$. For $X,Y\in\mathcal{L}_{\ell}(\mathcal{E},s)$,
we set 
\[
[X:Y]=\mathrm{length}(X/X\cap Y)-\mathrm{length}(Y/X\cap Y)
\]
where the length function is relative to the $\mathbb{Z}_{\ell}$-module
stucture if $u=g$ and to the $W$-module structure
if $u = s$ (here, $W := W(\Fbar_\ell)$ is the ring of Witt vectors of $\Fbar_\ell$). Then for all $X\in\mathcal{L}_{\ell}(\mathcal{E},g)$, 
\begin{equation}
\left[T_{\ell}(\mathcal{E}_g):X\right]=\left[T_{\ell}(\mathcal{E}_s) \colon \mathrm{red}_{\ell}(X)\right]\label{eq:DegComp}
\end{equation}
Indeed, we may choose a triple $(E,C,\phi)\in\mathcal{I}(\mathcal{E})$
such that \[
X=\phi\left(T_{\ell}((E/C)_g)\right)\quad\mbox{and}\quad\mathrm{red}_{\ell}(X)=\phi\left(T_{\ell}((E/C)_s)\right)\]
and then both sides of (\ref{eq:DegComp}) are equal to the exponent
of $\ell$ in the degree of $\phi$. 

\subsection{The supersingular case}
If $\mathcal{E}_{s}$ is a supersingular elliptic curve, then 
\[
\left(T_{\ell}(\mathcal{E}_s),F,V\right)\simeq\left(W^2, \pi_{\ell}\sigma,\pi_{\ell}\sigma^{-1}\right)\qquad\mbox{with}\qquad\pi_{\ell}={\mtwo 0 1 p 0} \in M_{2}(W)
\]
where $\sigma$ is the Frobenius automorphism of $W$. We thus obtain 
\[
\mathcal{L}_{\ell}(\mathcal{E}_s)\simeq\left\{ \pi_{\ell}^{i}W^2 \subset\mathcal{K}^{2} \colon i\in\mathbb{Z}\right\} \]
where $\mathcal{K}$ is the fraction field of $W$. In particular,
the map (\ref{eq:redp}) is uniquely determined by (\ref{eq:DegComp}).
Moreover, the endomorphism ring 
\[
\End\left(T_{\ell}(\mathcal{E}_s),F,V\right)\simeq\left\{ x\in M_{2}(W) \colon \pi\sigma(x)\pi^{-1}=x\right\} 
\]
is the maximal $\mathbb{Z}_{\ell}$-order $\mathcal{R}$ of a non-split
quaternion algebra $\mathcal{B}=\mathcal{R}\otimes_{\mathbb{Z}_{\ell}}\mathbb{Q}_{\ell}$
over $\mathbb{Q}_{\ell}$ and $\mathcal{B}^{\times}\simeq\mathcal{R}^{\times}\times\pi_{\ell}^{\mathbb{Z}}$
acts transitively on $\mathcal{L}_{\ell}(\mathcal{E},z)$. 
\begin{rem}
The reduction map (\ref{eq:redp}) is more difficult to analyze
when $\mathcal{E}_{s}$ is an ordinary elliptic curve, especially
when $\mathcal{E}_{g}$ does \emph{not }have complex multiplication. 
\end{rem}

\subsection{Matrices}
We choose a subgroup $T'$ of $T=\H_{1}(\mathcal{E}(\mathbb{C}),\mathbb{Z})$
such that $T/T'\simeq\mathbb{Z}/N\mathbb{Z}$. We put $V=T\otimes\mathbb{Q}$,
$B=\End_{\mathbb{Q}}(V)$ and 
\[
R=\left\{ x\in B \colon xT'\subset T'\mbox{ and }xT\subset T\right\} ,
\]
an Eichler order of level $N$ in $B\simeq M_{2}(\mathbb{Q})$. This
gives rise to identifications 
\[
\widehat{T}(\mathcal{E}_g)\simeq\widehat{T}(\mathcal{E},\mathbb{C})\simeq\widehat{T},\quad\widehat{V}(\mathcal{E}_g)\simeq\widehat{V}(\mathcal{E},\mathbb{C})\simeq\widehat{V}\quad\mbox{and}\quad\widehat{B}^{\times}/\widehat{R}^{\times}\simeq\mathcal{L}(\mathcal{E}, g), 
\]
where the last map sends $b$ to $(b\widehat{T}',b\widehat{T})$.
Suppose that $\mathcal{E}_{s}$ is supersingular and let $B_{\{\ell\}}=\End_{s}^{0}(\mathcal{E}_{s})$,
a definite quaternion algebra over $\mathbb{Q}$ with $\mathrm{Ram}_{f}(B_{\{\ell\}})=\{\ell\}$.
For $p\neq \ell$, the left action of $B_{\{\ell\},p}$ on $V_{p}(\mathcal{E}_s)\simeq V_{p}(\mathcal{E},g)\simeq V_{g}$
yields an isomorphism $\theta_{p}:B_{p}\simeq B_{\{\ell\},p}$,
and the left action of $B_{\{\ell\},\ell}$ on $V_{\ell}(\mathcal{E}_s)$ yields
an isomorphism $\theta_{\ell} \colon \mathcal{B}\simeq B_{\{\ell\},\ell}$ with $\mathcal{B}$
as above. Put $\mathrm{red}(\widehat{T}',\widehat{T})=(\widehat{T}'_{s},\widehat{T}_{s})$
and 
\[
R_{\{\ell\}}=\left\{ x\in B_{\{\ell\}} \colon x\widehat{T}'_{s}\subset\widehat{T}'_{s}\mbox{ and }x\widehat{T}_{s}\subset\widehat{T}_{s}\right\}.
\]
Thus $\theta_{p}(R_{p})=R_{\{\ell\},p}$ for all $p\neq \ell$,
and $\theta_{\ell}(\mathcal{R})=R_{\{\ell\},\ell}$ with $\mathcal{R}$ as
above. The map $b\mapsto(b\widehat{T}'_{s},b\widehat{T}_{s})$ yields
an identification $\widehat{B}_{\{\ell\}}^{\times}/\widehat{R}_{\{\ell\}}^{\times}\simeq\mathcal{L}(\mathcal{E},s)$,
and the reduction map 
\[
\widehat{B}^{\times}/\widehat{R}^{\times}\simeq\mathcal{L}(\mathcal{E},g)\stackrel{\mathrm{red}}{\longrightarrow}\mathcal{L}(\mathcal{E},s)\simeq\widehat{B}_{\{\ell\}}^{\times}/\widehat{R}_{\{\ell\}}^{\times}
\]
sends $\widehat{b}\widehat{R}^{\times}=\prod b_{p}R_{p}^{\times}$
to $\widehat{b}'\widehat{R}_{\{\ell\}}^{\times}=\prod b'_{p}R_{\{\ell\},p}^{\times}$
where $b'_{p}=\theta_{p}(b_{p})$ for $p\neq \ell$ and $b'_{\ell}=\theta_{\ell}(\pi_{\ell})^{v_{\ell}}$
for $v_{\ell}=\mathrm{ord}_{\ell}(\mathrm{nr}(b_{\ell}))$. We finally obtain a reduction map
\[
\Aut_{g}^{0}(\mathcal{E}_{g})\backslash\widehat{B}^{\times}/\widehat{R}^{\times}\simeq Y(\mathcal{E}_{g})\stackrel{\mathrm{red}}{\longrightarrow}Y(\mathcal{E}_{s})\simeq B_{\{\ell\}}^{\times}\backslash\widehat{B}_{\{\ell\}}^{\times}/\widehat{R}_{\{\ell\}}^{\times}\]
where $\End_{g}^{0}(\mathcal{E}_{g})$ embeds in $B$ through its
action on $V=\H_{1}(\mathcal{E}_{g}(\mathbb{C}),\mathbb{Q})$.

\subsection{The supersingular reduction of CM points}

We now assume that $\ell$ is inert in $K$, and let $\mathcal{E}$
be the N\'eron model over $\overline{\mathcal{O}}$ of an elliptic curve
with complex multiplication by $K$, i.e. $\End_{g}^{0}(\mathcal{E}_{g})=K$.
Then $\mathcal{E}_{s}$ is indeed supersingular, $Y(\mathcal{E}_{g})=\mathcal{CM}_{K}$
and $Y^{ss}(\mathcal{E}_{s})=Y^{ss}(\overline{\mathbb{F}}_\ell)$, the
set of supersingular points in $Y(\overline{\mathbb{F}}_\ell)$. We have
now identified the geometric reduction map
\begin{equation}
\mathrm{red} \colon \mathcal{CM}_{K}\rightarrow Y^{ss}(\overline{\mathbb{F}}_\ell)
\label{eq:FinalRed}
\end{equation}
with the adelic reduction map which we had previously considered.

\begin{rem}
The surjectivity of $(\ref{eq:FinalRed}$) implies that $Y^{ss}(\overline{\mathbb{F}}_\ell)\subset Y(\mathbb{F}[\infty])$.
On the other hand, class field theory shows that $\mathbb{F}[\infty]\simeq\mathbb{F}_{\ell^{2}}$.
We thus retrieve the well-known fact that $Y^{ss}(\overline{\mathbb{F}}_\ell)\subset Y(\mathbb{F}_{\ell^{2}})$. 
\end{rem}

\paragraph{The main correspondence between Heegner points and optimal embeddings.}   

Let now $c$ be a positive integer satisfying $(c, \ell N) = 1$ and let $\cO_c$ be the corresponding order in $K$. As a consequence of the above identifications, our Corollary~\ref{cor:main} implies the following: 
 
\begin{cor}\label{cor:corresp}
Let $s \in Y^{ss}(\Fbar_\ell)$ be a supersingular point. Choose $[\widetilde{E}, \widetilde{C}] = s$ and define 
$$
R_s = \End(\widetilde{E}, \widetilde{C}) = \left \{ \alpha \in \End(\widetilde{E}) \colon \alpha(\widetilde{C}) \subset \widetilde{C}\right \}. 
$$
There is a one-to-one correspondence between the following two sets:  
\[
\begin{array}{rcl}
\left \{
\begin{array}{c}
\textrm{Points } x \in \mathcal{CM}_K\textrm{ on } Y \\ 
\textrm{ of conductor }c
\textrm{ reducing to } s 
\end{array}
\right \} 
& 
\Longleftrightarrow
&
\left \{ 
\begin{array}{c}
R_s^\times-\textrm{conjugacy classes of} \\
\textrm{conjugate pairs of}\\ 
\textrm{optimal embeddings } \cO_{c} \hra R_s
\end{array}
\right \}.
\end{array}
\]
\end{cor} 
 
\begin{rem}
This was precisely the correspondence needed in \cite{jetchev-kane} to translate the equidistribution question for Heegner points to a question about optimal embeddings. It is shown in \cite[\S4.1]{jetchev-kane} that the latter relates to counting primitive representations of integers by ternary quadratic forms. 
For $c = 1$, the above correspondence is known as Deuring lifting theorem (see \cite{deuring}) and has been subsequently refined (as a correspondence) by Gross and Zagier \cite[Prop.2.7]{gross-zagier:singular}. 
\end{rem} 


\begin{rem}
The left-to-right map in Corollary~\ref{cor:corresp} is rather natural.
Let $(E,C)$ be an enhanced elliptic curve over $\overline{\mathbb{Q}}$
with complex multiplication by $K$ and coarse conductor $c$. Extend $(E, C)$ 
to an enhanced elliptic curve $(\mathcal{E},\mathcal{C})$ over
$\overline{\mathcal{O}}$ and suppose that $(\mathcal{E}_{s},\mathcal{C}_{s})\simeq(\widetilde{E},\widetilde{C})$.
A choice of isomorphisms $(\mathcal{E}_{s},\mathcal{C}_{s})\stackrel{\simeq}{\longrightarrow}(\widetilde{E},\widetilde{C})$ and $K\stackrel{\simeq}{\longrightarrow}\mathrm{End}^{0}(\mathcal{E},\mathcal{C})$ yields an embedding $\mathcal{O}_{c}\hookrightarrow R_{s}$,  
and the resulting $R_{s}^{\times}$-conjugacy class of pairs of conjugate
embeddings does not depend upon these two choices. Since
$\ell\nmid c$, it is still fairly straightforward to verify that the embeddings
thus obtained are optimal. What is not obvious is that this construction gives a one-to-one correspondence. This is precisely what we establish in our Theorem~\ref{thm:main} in a greater generality for quaternion algebras. 
\end{rem}

\begin{acknowledgements}
We would like to thank Dick Gross, David Jao, Ben Kane, Philippe Michel, Ken Ribet and Nicolas Templier for helpful discussions.
\end{acknowledgements}

\bibliographystyle{amsalpha}
\bibliography{biblio-math}

\end{document}